\documentclass[a4paper,twoside,11pt,reqno]{amsart}
\usepackage{amsmath,amsthm,amssymb,enumerate,pgf,cite}
\usepackage{tikz}
\usetikzlibrary{quotes,angles}
\usepackage{dsfont}

\usepackage{setspace}

\newtheorem{thm}{Theorem}[section]

\newtheorem{lemma}[thm]{Lemma}

\theoremstyle{definition}
\newtheorem{defn}[thm]{Definition}
\theoremstyle{remark}
\newtheorem{rem}[thm]{Remark}
\numberwithin{equation}{section}
\makeatletter
\def\bea{\begin{eqnarray*}}
\def\eea{\end{eqnarray*}}

\let\TeXchi\chi
\def\chi{{\setbox0 \hbox{\mathsurround0pt
$\TeXchi$}\hbox{\raise\dp0 \copy0 }}}
\begin{document}

\title{The Gevrey class of the Euler-Bernoulli beam model with singularities}%

\author[J.E. Mu\~{n}oz Rivera]{Jaime E. Mu\~{n}oz Rivera}%
\address{Universidad del B\'\i o B\'\i o Collao 1202, Casilla 5C, \newline
\indent Concepci\'{o}n - Chile}

\address{Laboratório Nacional de Computação Científica (LNCC) , \newline
\indent Petr\'{o}polis, Brasil}

\indent
\email{jemunozrivera@gmail.com}%

\author[M.G. Naso]{Maria Grazia Naso}

\address{DICATAM,
Universit\`{a} degli Studi di Brescia, Via Valotti 9, 25133
Brescia,
Italia}%
\email{mariagrazia.naso@unibs.it}%

\author[B.T. Silva Sozzo]{Bruna T. Silva Sozzo}

\address{Universidade Norte do Paran\'{a}, Londrina, Brasil}%
\email{bruna.sozzo@uenp.edu.br}%

%
\begin{abstract}
    We study  the Euler-Bernoulli beam model with singularities  at the points $x=\xi_1$, $x=\xi_2$ and with localized viscoelastic dissipation of Kelvin-Voigt type. We assume that the beam is composed by two materials; one is an elastic material and the other one is a viscoelastic material of Kelvin-Voigt type. 
    
    Our main result is that the corresponding semigroup is  immediately differentiable and also of  Gevrey class $4$. 
    In particular, our result implies that the model is exponentially stable, has the linear stability property, and  the smoothing effect property over the initial data.
\end{abstract}
\maketitle


\section{Introduction}
 In this work we consider the Euler-Bernoulli beam equation of lenght $\ell$ consisting of three components. Introducing $0<\ell_0<\ell_1<\ell$, components $(0,\ell_0)$ and $(\ell_1,\ell)$  are made by an elastic material 
 whereas component $(\ell_0,\ell_1)$ is viscoelastic of  Kelvin-Voigt type  \cite{chawla}. 
 Furthermore, we consider a point actuator that applies a localized force or moment at a specific point along the beam, rather than a distributed effect along its length. We can understand that the point force represents, for example, a small piezoelectric patch or stack attached to the beam, which exerts a point force that pushes perpendicularly on the beam, while the concentrated moment represents, for example, a small piezoelectric patch that induces a localized bending moment when voltage is applied.
The piezoelectric effect generates this force or moment proportional to the applied voltage mediated by the piezoelectric coefficient.
The pointwise loading is modeled with a Dirac delta function:
$q(x)=F_p\delta(x-\xi)$ while the point moment’s effect on the curvature is modeled by $-\frac{d}{dx}[M_p\delta(x-\xi)]$, where $\delta$ stands for the Dirac delta function.  
    The following Fig.~\ref{fig01} describes the situation.

		\noindent 	%
		\begin{figure}[h!]
			\setlength{\unitlength}{2pt}
			\begin{center}
				\begin{tikzpicture}
					\draw[fill=yellow, opacity=0.3] (0,-0.3) rectangle (2.97,0.3); 
					\draw[fill=red!40, opacity=0.3] (3.03,-0.3) rectangle (4.97,0.3);
					\draw[fill=yellow, opacity=0.3] (5.03,-0.3) rectangle (8,0.3);
					\draw[->, very thick] (-1,0)--(9,0);
					\draw[->, very thick] (0,-1.2)--(0,1);
					\draw[->, very thick] (1.2,-0.5)--(1.2,0);
                    \draw[->, very thick] (4.5,-0.5)--(4.5,0);					
                    \draw[fill=gray, opacity=0.3, bottom color=black, top color= white] (-0.2,-0.7) rectangle (0,0.7);
					\draw[fill=gray, opacity=0.3, bottom color=black, top color= white] (8,-0.7) rectangle (8.2,0.7);
                    \draw[-] (-0.2,-0.7) node[below] {$0$};
					\draw[-] (1,-0.3) node[below] {$\xi_1$};
					\draw[-] (3,-0.3) node[below] {$\ell_0$};
                    \draw[-] (4.3,-0.3) node[below] {$\xi_2$};
					\draw[-] (5,-0.3) node[below] {$\ell_1$};
					\draw[-] (8,-0.7) node[below] {$\ell$};
					\draw[-] (3,0.15) node {$\bullet$};
					\draw[-] (3,-0.15) node {$\bullet$};
					\draw[-] (5,0.15) node {$\bullet$};
					\draw[-] (5,-0.15) node {$\bullet$};
					\draw[|-|,color=black, very thick] (0,0.5)--(2.97,0.5);
					\draw[|-|,color=black, very thick] (3.03,0.7)--(4.97,0.7);
					\draw[|-|,color=black, very thick] (5.03,0.5)--(8,0.5);
					\draw[-] (1.5, 0.6) node[above] {{\it Elastic part}};
					\draw[-] (4, 1.2) node[above] {{\it Viscoelastic}};
					\draw[-] (4, 0.8) node[above] {{\it part}};
					\draw[-] (6.5, 0.6) node[above] {{\it Elastic part}};
				\end{tikzpicture}
			\end{center}
			\vspace{-0.7cm}
			\caption{An Euler-Beroulli beam constructed by three components}\label{fig01}
		\end{figure}
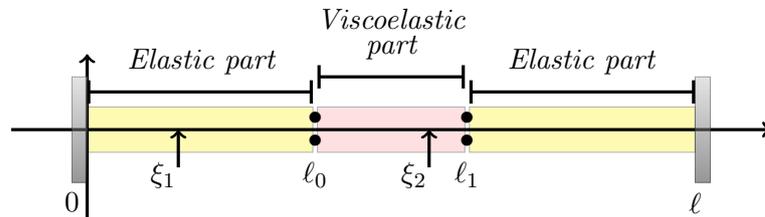

\noindent Let $0<T \leq +\infty$, the corresponding model with initial and clamped boundary conditions is given by 
    \begin{eqnarray} \label{probPuro}
        \left\{
        \begin{array}{ll}
        \rho u_{tt} + \left(\alpha  u_{xx} + \kappa  u_{txx}\right)_{xx} +\gamma_1 u_t\delta_{\xi_1}+\gamma_2 u_t\delta_{\xi_2} +\gamma_3 u_{xt}\dfrac{\partial \delta_{\xi_2} }{\partial  x}= 0 \ \ 
        &\text{in $(0,\ell) \times (0,T)$},\\
        \noalign{\medskip}
        u(0,t) = u_x(0,t) = u(\ell,t) = u_x(\ell,t) = 0, &\text{in $(0,T)$}, \\
        \noalign{\medskip}
        u(x,0) = u_0(x), \ \ u_t(x,0) = u_1(x),  &\text{in $(0,\ell)$},
        \end{array}
        \right.
    \end{eqnarray}
    where $ u(\cdot,\cdot)$ 
    represents the transversal displacement of the beam, and
    $u_0$ and $u_1 : (0, \ell)\to \mathbb{R}$ are assigned functions.
   Moreover, we assume that $\gamma_i$, $i=1,2,3$, are positive constants and
    \begin{eqnarray*}
       && \rho(x) :=\left\{
        \begin{array}{ll}
        \rho_1 & \text{if $x \in (0,\ell_0)\cup (\ell_1,\ell)$,} \\
        \noalign{\medskip}
        \rho_2  & \text{if $  x \in (\ell_0,\ell_1)$,}
        \end{array}
        \right.
        \\
        \noalign{\medskip}
        &&
         \alpha(x): =\left\{
        \begin{array}{lll}
        \alpha_1  & \text{if $x \in (0,\ell_0)\cup (\ell_1,\ell)$,}\\
        \noalign{\medskip}
        \alpha_2  & \text{if $  x \in (\ell_0,\ell_1)$,}
        \end{array}
        \right.
        \quad
        \kappa(x) :=\left\{
        \begin{array}{lll}
        0  & \text{if $x \in (0,\ell_0)\cup (\ell_1,\ell)$,}\\
        \noalign{\medskip}
        \alpha_0  & \text{if $  x \in (\ell_0,\ell_1)$,}
        \end{array}
        \right.
    \end{eqnarray*} 
    with  $\rho_1, \rho_2, \alpha_1, \alpha_2$ and  $\alpha_0$  positive constants.
   Function $\delta_\xi$ is the Dirac mass $+1$ at the point $x=\xi\in (0,\ell)$.
    Denoting by  $$ M(x,t) := \alpha(x) u_{xx}(x,t) + \kappa(x) u_{txx}(x,t),$$ the solution must satisfy the following transmission conditions at the interfaces $x=\ell_i$ with $i=0,1$ and in $(0,T)$
   \begin{eqnarray} \label{xxx}
   \begin{array}{l}
        u(\ell_i^-,t) = u(\ell_i^+,t),\quad  u_x(\ell_i^-,t) = u_x(\ell_i^+,t),\\
        \noalign{\medskip}
         M(\ell_i^-,t) = M(\ell_i^+,t),\quad M_x(\ell_i^-,t) = M_x(\ell_i^+,t).
         \end{array}
    \end{eqnarray} 
The above model has received considerable attention in the past and it was studied by several authors when $\gamma_1=\gamma_2=\gamma_3=0$.  I. Lasiecka \cite{92} considered the Euler-Bernoulli model with boundary dissipation effective only at the moment and she proved the exponential stability of the corresponding semigroup. 
In \cite{LiuChen} it was prove that the semigroup associated with the Euler-Bernoulli beam model with global viscoelastic damping is analytic and exponentially stable.
Liu et al. in \cite{LiuLiu} studied a localized viscoelastic damping and they  showed that the corresponding semigroup is exponentially stable but not  analytical.
In \cite{irena}  I. Lasiecka and B. Belinskiy analyzed the Euler-Bernoulli model with non dissipative boundary condition  of the type $u_{xx}(\ell)=-\kappa u_{xt}(\ell)$ and they proved that the corresponding semigroup is of Gevrey regularity order $\delta>2$. 
The method there used is based on microlocal analysis.
Caggio and Dell'Oro \cite{caggio} considered an Euler-Bernoulli beam equation with localized discontinuous structural damping and they proved that the associated $C_0$-semigroup $(S(t))_{t \ge 0}$ is of Gevrey class $\delta >24$ for $t>0$, hence immediately differentiable. Moreover, they showed that $(S(t))_{t \ge 0}$ is exponentially stable.

Our model presents singularities defined by the Dirac delta function and its derivative. In this context, it is important to determine the maximum regularity that the solution of the model possesses. 
For this reason we study this topic here.

Let us introduce some concepts that will be used throughout the text.
 First, we provide a definition of Gevrey class semigroups \cite{Gevrey}.
\begin{defn} \label{defGreve}    
    A strongly continuous semigroup $e^{At}$ is of Gevrey class $\delta>0$ for $t>t_0$  if $e^{At}$ is infinitely differentiable for $t>t_0$ and for every compact $K \in (t_0 ,\infty)$ and each $\theta>0$, there exists a constant $C = C(K,\theta)$ such that
    $$
   \left \| \left[e^{At}\right]^{(n)}\right\|_{\mathcal{L}(\mathcal{H})}\leq C\theta^n(n!)^\delta,\qquad \forall t\in K,\quad n=0,1,2,\cdots .
     $$
\end{defn}
\begin{rem}
 Gevrey regularity is described in terms of the bounds on all derivatives of the
semigroup. These bounds are weaker than the corresponding ones corresponding to characterization of analyticity, but they are stronger than the ones corresponding to differentiability (see \cite{Chen-Triggiani,pazy,taylor}). In other words, the  Gevrey class semigroup has more regular properties than a differentiable semigroup, but less regular than an analytical semigroup. The Gevrey functions of the $\delta$-class $(\delta > 1)$ tell us the degree of divergence of the corresponding Taylor expansion. The larger the $\delta$-class, the faster the divergence. 
If we denote by $G^\delta$ the Gevrey class for some fixed $\delta \ge 1$, the scale of Gevrey classes is a nested scale of the parameter $\delta$ that fills the gap between analytic $\mathfrak{A}$ and $C^\infty$ functions:
\[
\mathfrak{A}= G^1\subset G^\delta\subset G^{\delta'} \subset C^\infty, \quad \text{for $1<\delta<\delta'$.}
\]
The case of analytic functions (i.e. $\mathfrak{A}= G^1$ Gevrey-$\delta$-class function with $\delta=1$) is the simplest and arguably most important of this scale.
\end{rem}
Subsequently, Taylor (e.g. \cite{taylor}) gave a sufficient condition for a strongly continuous semigroup to be
of Gevrey-$\delta$-class.
\begin{thm} \label{defGe} 
    Let $S = (S(t))_{t \geq 0}$ be a strongly continuous and bounded semigroup on a Hilbert space $X$.  Suppose that the infinitesimal generator $A$ of the semigroup $S$ satisfies the following estimate, for some $0< \mu <1$:
    \begin{eqnarray} \label{condClasG}
        \lim_{\lambda \in \mathbb{R}, \ |\lambda| \rightarrow \infty} \sup |\lambda|^{\mu}\|(i \lambda - A)^{-1} \| < \infty.
    \end{eqnarray}
    Then $S = (S(t))_{t \geq 0}$ is Gevrey-$\delta$-class for $t > 0$, for every $\delta > \dfrac{1}{\mu}$.
\end{thm}
\begin{rem}
 The Gevrey  rate $\dfrac{1}{\mu}>1$ 'measures' the degree of  divergence of its  power series. 
\end{rem}
    
\begin{rem}    
    The Gevrey semigroup is equivalent to semigroup of class $D^p$ ($p=\frac 1\mu$) defined in \cite{Crand-Pazy}. 
    Theorem \ref{defGe} is equivalent to Theorems 2.1, 2.2, 2.3 of \cite{Crand-Pazy}. Thus, the Gevrey
class semigroup has a stronger regularity property than the differentiable semigroup,
but is weaker than the analytic semigroup  (e.g., \cite{pazy,taylor}).
\end{rem}
Our  main result  is to show that the semigroup ${(S(t))}_{t\ge0}$, associated with model \eqref{probPuro}, is of  Gevrey class $4$. This implies several important results, first that the semigroup is  immediately differentiable, which 
implies the instantaneous smoothing effect property on the initial data over $D(\mathcal{A}^\infty)$. 
Moreover, our proof implies that the semigroup is exponentially stable and that the type of the semigroup is equal to the upper bound of  spectrum of its infinitesimal generator $\mathcal{A}$. That is, the semigroup enjoys the  linear stability property. Finally, the semigroup is instantaneous uniformly continuous operator.

The remaining part of this manuscript is organizing as follow. In Section~\ref{sec-C_0_semigroup} we provide some
notation and standing assumptions, then we establish  the well posedness of the model by showing that  a $C_0$ semigroup of contractions is found on a well-defined Hilbert space $\mathcal{H}$. 
In  Section~\ref{sec-Gevrey} we show the main result of this paper:  the semigroup is of  Gevrey class $4 $.

\section{Preliminaries and semigroup setting}\label{sec-C_0_semigroup}
To use the semigroup approach, we switch to a transmission problem, and, to do that, we consider 
$$
I:=(0,\xi_1)\cup(\xi_1,\ell_0)\cup(\ell_0,\xi_2)\cup(\xi_2,\ell_1)\cup(\ell_1,\ell).
$$ 
To facilitate notations, let us introduce 
\newcommand{\salto}[2]{[\![#1]\!]_{\xi_{#2}}}
 the  jump of $f$ in $\xi$ as 
$$
\salto{f}{}:=f(\xi^+)-f(\xi^-).
$$
From now on and without loss of generality, we assume $\rho \equiv 1$. Therefore, in this case, it is easy to see that system \eqref{probPuro} is equivalent to
\begin{eqnarray}
u_{tt} + \left[\alpha(x) u_{xx} + \kappa(x) u_{txx}\right]_{xx}&= 0 \quad \text{in $ I \times (0,T)$,}\label{eq21}
\end{eqnarray}
satisfying the same initial and boundary conditions written in  \eqref{probPuro}
and supplemented with 
 the transmission  conditions  
\begin{eqnarray}\label{wxxO}
&& \salto{u}{1}=\salto{u_x}{1}=\salto{u_{xx}}{1}=0,\quad \salto{u}{2}=\salto{u_x}{2}=0,\\
\noalign{\medskip}
\label{wxx}
&& \salto{\alpha u_{xxx}}{1}=-\gamma_1u_t(\xi_1,t),\quad  \salto{\alpha u_{xxx}}{2}=-\gamma_2u_{t}(\xi_2,t),\quad  \salto{\alpha u_{xx}}{2}=\gamma_3u_{xt}(\xi_2,t),
\end{eqnarray}
together with \eqref{xxx}.

   Let the energy functional $\mathcal{E}$, associated with the system \eqref{eq21}--\eqref{wxx}, be given by	
    $$
        \mathcal{E}(t) = \frac{1}{2}\int_{0}^{\ell}\left(|u_{t}|^{2} + \alpha|u_{xx}|^{2}\right)dx.
    $$
   { A straightforward calculation yields:
\[
\frac{d}{dt} \mathcal{E}(t) = -\int_{\ell_0}^{\ell_1} \alpha_0 |u_{xxt}|^2 \, dx - \gamma_1 |u_t(\xi_1, t)|^2 - \gamma_2 |u_t(\xi_2, t)|^2 - \gamma_3 |u_{xt}(\xi_2, t)|^2.
\]

We observe that the coefficient \(\kappa(x) = \alpha_0>0\) in the interval \((\ell_0, \ell_1)\) and vanishes elsewhere.
}

Moreover, note that problem \eqref{probPuro}--\eqref{xxx} is equivalent to \eqref{eq21}--\eqref{wxx} because the corresponding weak formulations (variational equations) are the same. 

    Let us define the phase space $\mathcal{H}$ by
$$
\mathcal{H}:=H_0^2(0,\ell)\times L^2(0,\ell).
$$
 It is easy to see that $\mathcal{H}$ is a Hilbert space with the norm
		\begin{align*}
    		\|U\|_{\mathcal{H}}^{2} = \int_{0}^{\ell}\left(|v|^{2} + \alpha|u_{xx}|^{2}\right)dx,
		\end{align*}
    where  $U = (u,v)^{\top}\in \mathcal{H}$ with $v := u_t$. 
 Therefore system \eqref{eq21}--\eqref{wxx}  can be written
in the form of an abstract first-order evolutionary Cauchy problem
\begin{align}\label{semigroup 3.1}
\begin{array}{rc}
		 U_t =  \mathcal{A}U,\quad U(0) =  U_0,
	\end{array} 
\end{align}   
   where $U_0=(u_0, u_1)^\top$ and the operator $\mathcal{A}\colon \mathcal{D}  (\mathcal{A})\subset \mathcal{H}\to\mathcal{H}$ is given by
		\begin{align}\label{opA}
    		\mathcal{A}U =
    		\begin{pmatrix}
        		v\\
        		-\left(\alpha u_{xx}+\kappa v_{xx}\right)_{xx} 
    		\end{pmatrix}.
		\end{align}
Here we consider the same boundary conditions written in  \eqref{probPuro}
\begin{eqnarray}\label{DwxxO1}
 u(0) = u_x(0) = u(\ell) = u_x(\ell) = 0,
\end{eqnarray}
the transmission conditions due to the material
\begin{eqnarray}\label{DwxxO2}
         M(\ell_i^-) = M(\ell_i^+),\quad M_x(\ell_i^-) = M_x(\ell_i^+),
\end{eqnarray}
and, under the above notation, the transmission conditions due to the pointwise sources, namely 
\begin{eqnarray}
&&\salto{u_{xx}}{1}=0,  \quad \salto{\alpha u_{xxx}}{1}=-\gamma_1v(\xi_1),\quad  \\
\noalign{\medskip}
&&\salto{\alpha u_{xxx}}{2}=-\gamma_2v(\xi_2),\quad  \salto{\alpha u_{xx}}{2}=\gamma_3v_{x}(\xi_2). \label{DwwO3}
\end{eqnarray}
The domain of the operator $\mathcal{A}$ is
		\begin{align*}
    		\mathcal{D}(\mathcal{A}) = \left\lbrace U = (u,v)^{\top}\in \mathcal{H}\colon u,v\in H^{2}_{0}(0,\ell),\, M \in H^{2}(I),\, \text{verifying \eqref{DwxxO2}--\eqref{DwwO3}}\right\rbrace.
		\end{align*}
		Over the above conditions, we have that  $\mathcal{A}$ is dissipative, that is 
    \begin{equation}
    	\mbox{\rm Re} \left \langle \mathcal{A}U,U \right \rangle_{\mathcal{H}}
    	 = -\int_{\ell_0}^{\ell_1}\alpha_0|v_{xx}(x)|^{2} dx-\gamma_1|v(\xi_1)|^2-\gamma_2|v(\xi_2)|^2 -\gamma_3|v_{x}(\xi_2)|^2  \leq  0 .\label{dissi}
    \end{equation}

For $\lambda \in \mathbb{R}$, $U\in  D(\mathcal{A})$ and $F\in  \mathcal{H}$ the resolvent equation is given by 
\begin{equation}\label{Resolv}
i\lambda U - \mathcal{A}U = F,
\end{equation}
that, in terms of the components,  can be written as 
    		\begin{eqnarray}\label{reso}
    \begin{array}{ll}
		\label{Eqrtt1}	i \lambda u - v  = f_{1} & \text{in $H^{2}(0,\ell)$},\\
\noalign{\medskip}
		\label{Eqrtt2}	i \lambda  v + \left[\alpha(x) u_{xx} + \kappa v_{xx}\right]_{xx}  =  f_{2}  & \text{in $L^{2}(0, \ell)$} ,
		\end{array}
\end{eqnarray}
		verifying \eqref{DwxxO1}--\eqref{DwwO3}. 
Multiplying \eqref{Resolv} by $U$, taking real part and using \eqref{dissi}, we arrive to 
    \begin{equation}\label{uno}
\int_{\ell_0}^{\ell_1}\alpha_0|v_{xx}(x)|^{2} dx+\gamma_1|v(\xi_1)|^2+\gamma_2|v(\xi_2)|^2 +\gamma_3|v_{x}(\xi_2)|^2 = \mbox{\rm Re} \left \langle F,U \right \rangle_{\mathcal{H}}.
    \end{equation}
    
 Let  $\sigma(\mathcal{A})$ be the spectrum of the operator $\mathcal{A}$, and  $\varrho(\mathcal{A}):=\mathbb{C}\setminus \sigma(\mathcal{A}) $  is the resolvent set of $\mathcal{A}$.   
 To show that $\mathcal{A}$ is the infinitesimal generator of a contraction semigroup
we use the following result.
\begin{thm}
  \label{TDenso}
Let $\mathcal{A}$ be dissipative with $0\in\varrho(\mathcal{A}).$  If $\mathcal{H}$ is reflexive then
$\mathcal{A}$ is the infinitesimal generator of a semigroup of contractions.
\end{thm}
\begin{proof}
Since $\varrho(\mathcal{A})$ is an open set we have that there exists $\epsilon>0$ such that $\epsilon \in \varrho(\mathcal{A})$. This implies that  any $\lambda>0$ belongs to $ \varrho(\mathcal{A})$, in particular we have that
$R(I-\mathcal{A}) = \mathcal{H}$. Using  \cite[Theorem 4.6]{pazy},
we conclude that $\overline{D(\mathcal{A})} = \mathcal{H}$. Applying Lummer-Phillips Theorem \cite[Theorem~1.4.3]{pazy}, our conclusion follows.
\end{proof}

So,
we obtain that
\begin{thm}
  The operator $\mathcal{A}$  is the infinitesimal generator of  a $C_0$-semigroup, $e^{\mathcal{A} t}$, of contractions on $\mathcal{H}$ . Moreover, for any $U_0 \in  \mathcal{H}$  the solution 
  $U \in  C(0, T ; \mathcal{H} )$. If $U_0 \in  \mathcal{D}(\mathcal{A})$,
then $U \in  C^1(0, T ; \mathcal{H}  ) \cap  C(0, T ; \mathcal{D}(\mathcal{A}))$.
    \end{thm}
\begin{proof}
Since $\mathcal{A}$ is dissipative and $\mathcal{H}$ is reflexive, by Theorem \ref{TDenso}, 
it is enough to prove that $0\in\varrho(\mathcal{A}).$
In fact, we show that for any  $F=(f_1,f_2)^\top\in \mathcal{H}$, there exists only one  $U\in D(\mathcal{A})$ such that $-\mathcal{A}U=F$,	
From \eqref{reso}$_1$  we have $-v = f_1 \in H^2_0$. The second component of the system is of the form $-\alpha u_{xxxx}= f_2+\alpha_0 f_{1,xxxx}$.
Using standard procedures (see, e.g., \cite{LiuChen,LiuLiu}) we can conclude that $U=(u,v)^\top\in D(\mathcal{A})$.
Finally, the density of $\mathcal{D}(\mathcal{A})$ follows from \cite[Theorem 4.6]{pazy}.
\end{proof}

\section{Gevrey class}\label{sec-Gevrey}
{ 
In this section, we show that the semigroup generated by the system \eqref{probPuro} belongs to the Gevrey class of order $\delta > 4$.

The main difficulty in establishing the Gevrey regularity of the semigroup stems from the lack of suitable observability inequalities within the viscous region. In particular, it is not possible to estimate the boundary terms at the endpoints of the viscous interval solely in terms of the viscous dissipation. To overcome this obstacle, we employ Lemma~\ref{Funda}, which provides a crucial relation between the second- and third-order derivatives at the interfaces of the elastic subintervals. This relation allows us to recover the necessary boundary estimates and thus to control the elastic component of the solution effectively.

Our analysis is carried out on the subintervals $(0, \xi_1)$, $(\xi_1, \ell_0)$, $(\ell_0, \xi_2)$, and $(\xi_2, \ell_1)$. Unless otherwise stated, the notation $\|\cdot\|$ refers to the $L^2$-norm over the corresponding interval, while $\|\cdot\|_{H^m}$ denotes the $H^m$-norm on that same domain.

In developing our arguments, we make use of several classical results from functional analysis and Sobolev space theory, as presented in Adams~\cite{adams}.
}


       \begin{lemma}[{\cite[Theorem 5.8, p. 139]{adams}}]\label{desDI}
        Let $a,b \in \mathbb{R}$. If $w, w^{(m)} \in L^2(a,b)$, then, $w^{(j )} \in L^2(a, b)$, for $j =1,\ldots , m$, where $w^{(j )} = \dfrac{d^j}{dx^j} w$. Then, 
        \begin{eqnarray*}
            \|w^{(j)}\|_{L^2} \leq C\|w\|_{L^2} + C \|w\|^{1-\frac{j}{m}}_{L^2}\|w^{(m)}\|^{\frac{j}{m}}_{L^2}. \ \ 
        \end{eqnarray*}
    \end{lemma}
    \begin{lemma}[{\cite[Theorem 5.9, p. 140]{adams}}] \label{desGN}
        Let $w \in H^2(a,b)$ and $s \in [a,b]$. Then we have, 
        \begin{eqnarray*}
            |w(s)| \leq  C\|w\|_{L^2} + C\|w\|^{\frac{3}{4}}_{L^2}\|w_{xx}\|^{\frac{1}{4}}_{L^2}  
            \ \ \text{and} \ \
            |w_{x}(s)| \leq C\|w\|_{L^2} + C \|w\|^{\frac{1}{4}}_{L^2}\|w_{xx}\|^{\frac{3}{4}}_{L^2}.
        \end{eqnarray*}
    \end{lemma}
Our starting point is to show the strong
stability, which we establish in the following Lemma.
    \begin{lemma}\label{Imag} 
    Let $\mathcal{A}$ be the infinitesimal generator given by \eqref{opA} then 
 $\varrho(\mathcal{A}) \supseteq \{ i \lambda; \ \lambda \in \mathbb{R} \} \equiv i \mathbb{R}.$
      \end{lemma}    
        \begin{proof}
                Let us denote by 
\begin{equation*}
\mathcal{N}=\{s\in \mathbb{R}^+\colon (-is,\,is)\subset \varrho(\mathcal{A})\}.
\end{equation*}
Since $0\in \varrho(\mathcal{A})$, $\mathcal{N}\neq \varnothing$. 
Putting $\lambda=\sup \mathcal{N}$, we verify that if $\lambda=+\infty$ it follows that $i\mathbb{R}\subseteq \varrho(\mathcal{A})$.
By contradiction, suppose  that  $\lambda<\infty$, then there exists a sequence $\{\lambda_n\}\subseteq \mathbb{R}$ such that $ \lambda_n\to\lambda< \infty$ and
$$\|(i\lambda_nI-\mathcal{A})^{-1}\|_{\mathcal{L}(\mathcal{H})}\to \infty.$$
Hence, there exists a sequence $\{f_n\}\subseteq \mathcal{H}$ verifying  $\|f_n\|_{\mathcal{H}}=1$ and $\|(i\lambda_nI-\mathcal{A})^{-1}f_n\|_{\mathcal{H}}\to \infty$. Denoting by
$$
\tilde{U}_n=(i\lambda_nI-\mathcal{A})^{-1}f_n\quad \Rightarrow \quad f_n=i\lambda_n\tilde{U}_n-\mathcal{A}\tilde{U}_n,
$$
and $U_n=\dfrac{\tilde{U}_n}{\|\tilde{U}_n\|_{\mathcal{H}}}$, $F_n=\dfrac{f_n}{\|\tilde{U}_n\|_{\mathcal{H}}}$, we conclude that $U_n$ verifies $\|U_n\|_{\mathcal{H}}=1$ and 
\begin{equation*}\label{convresol}
i\lambda_nU_n-\mathcal{A} U_n=F_n\to 0.
\end{equation*}
Since $\|\mathcal{A}  U_n\|_{\mathcal{H}}\leq C$, it follows that $U_n$ is bounded in $D(\mathcal{A})$. In particular, we find that 
\begin{eqnarray}\label{bbb}
\alpha u_{n,xx}+\kappa v_{n,xx} \quad \text{is bounded in $H^{2}(0,\ell)$}.
\end{eqnarray}
From \eqref{Eqrtt1}  and by dissipation \eqref{uno}, we obtain
$$
(u_n,v_{n})\rightarrow (0,0)\quad \mbox{strongly in}\quad [L^2(\ell_0,\ell_1)]^2.
$$
Therefore, from \eqref{bbb}, there exists a subsequence of $U_n$ (we still denote in the same way) such that 
$$
U_n\quad \rightarrow\quad U=(u,v)\quad \mbox{strongly in}\quad \mathcal{H}.
$$
So we have that \(\|U\|_{\mathcal{H}} = 1\). Since the operator \(\mathcal{A}\) is closed and the sequence \(\mathcal{A} U_n = i \lambda_n U_n - F_n\) converges strongly in \(\mathcal{H}\) to \(i \lambda U\) (noting that \(F = 0\)), it follows that \(U\) satisfies:
\[
i \lambda U - \mathcal{A} U = 0.
\]
From equation \eqref{uno} with \(F = 0\), we obtain:
\[
v_{xx} = 0 \quad \text{for all } x \in (\ell_0, \ell_1), \quad v(\xi_1) = v(\xi_2) = v_x(\xi_2) = 0,
\]
where \(\xi_1, \xi_2 \in (\ell_0, \ell_1)\). Similarly, from equation \eqref{Eqrtt1}, we deduce:
\[
u_{xx} = 0 \quad \text{for all } x \in (\ell_0, \ell_1), \quad u(\xi_2) = u_x(\xi_2) = 0,
\]
with \(\xi_2 \in (\ell_0, \ell_1)\). Since \(u_{xx} = 0\) and \(v_{xx} = 0\) in \((\ell_0, \ell_1)\), and given the boundary conditions \(u(\xi_2) = u_x(\xi_2) = 0\) and \(v(\xi_1) = v(\xi_2) = v_x(\xi_2) = 0\), it follows that \(u = 0\) and \(v = 0\) in \((\ell_0, \ell_1)\).

Consequently, in the interval \((0, \ell_0)\), the system reduces to:
\[
\begin{cases}
i \lambda u - v = 0, \\
i \lambda v + \left[ \alpha(x) u_{xx} \right]_{xx} = 0,
\end{cases}
\]
with boundary conditions \(u(\ell_0) = u_x(\ell_0) = u_{xx}(\ell_0) = u_{xxx}(\ell_0) = 0\). This implies that \(u = 0\) in \((0, \ell_0)\). Applying a similar argument in the interval \((\ell_1, \ell)\), we also conclude that \(u = 0\) in \((\ell_1, \ell)\). Thus, \(U = (u, v) = 0\) in \(\mathcal{H}\).

However, this contradicts the assumption that \(\|U\|_{\mathcal{H}} = 1\). Therefore, our conclusion follows.

        \end{proof}
 
        \subsection{Estimates over the viscoelastic component of interval $(\ell_0, \ell_1)$}
        
Let $v_1 \in H^4(\ell_0,\ell_1)\cap H_0^2(\ell_0,\ell_1)$ be the solution of 
\begin{equation}\label{eqv1}
           i\lambda v_1 + (\alpha_0 v_{1,xx})_{xx} =  f_2 \in L^2(\ell_0,\ell_1).
\end{equation}
           To facilitate estimates over $(\ell_0,\ell_1)$, let us introduce $v_2 = v-v_1$, where $v$ is the solution of   \eqref{Eqrtt1}--\eqref{Eqrtt2}. Hence we have that $v=v_1+v_2$. Note that $v_2\in H^2(\ell_0,\ell_1)$ verifies 
        \begin{eqnarray}
        \begin{array}{l} \label{eqv2}
             i \lambda v_2 + (\alpha u_{xx})_{xx} + (\alpha_0 v_{2,xx})_{xx} = 0,\\
             \noalign{\medskip}
\salto{\alpha u_{xxx}}{2}=-\gamma_2v(\xi_2),\quad  \salto{\alpha u_{xx}}{2}=\gamma_3v_{x}(\xi_2). 
        \end{array}
        \end{eqnarray}
    Making inner product from \eqref{eqv1} by $\overline{v_{1,xxxx}}$ and using the boundary and transmission conditions, it follows that
        \begin{eqnarray} 
            \int_{\ell_0}^{\ell_1}  i \lambda  |v_{1,xx}|^2 dx  +\int_{\ell_0}^{\ell_1} \alpha_0 |v_{1,xxxx}|^2 dx = \int_{\ell_0}^{\ell_1}  f_2 \overline{v_{1,xxxx}} dx. \label{eqv1xxxx}
        \end{eqnarray}
        Considering the real part in \eqref{eqv1xxxx}, we have
        \begin{eqnarray*}
           &\displaystyle \int_{\ell_0}^{\ell_1} \alpha_0 |v_{1,xxxx}|^2 dx = \mbox{Re} \left( \displaystyle \int_{\ell_0}^{\ell_1}  f_2 \overline{v_{1,xxxx}} dx \right)) \quad
           \Rightarrow \quad \|v_{1,xxxx}\| \leq C\|F\|_{\mathcal{H}}.
        \end{eqnarray*}
        Taking  the imaginary part in \eqref{eqv1xxxx} and using the above inequality, we get 
        \begin{eqnarray} \label{v1xx}
            \sqrt{ |\lambda|} \|v_{1,xx}\| \leq C \|F\|_{\mathcal{H}}.
        \end{eqnarray}
        Similarly, making inner product between   \eqref{eqv1} and $\overline{i\lambda v_{1}}$, we find  
        \begin{eqnarray*}
           &\displaystyle \int_{\ell_0}^{\ell_1}  \lambda^2 |v_{1}|^2 dx  \leq C \|F\|_{\mathcal{H}}^2.
        \end{eqnarray*}
        Therefore, we have 
        \begin{eqnarray}\label{v1}
            |\lambda |^2\|v_{1}\|^2 +  |\lambda |\|v_{1,xx}\|^2 + \|v_{1,xxxx}\|^2 \leq C \|F\|_{\mathcal{H}}^2.
        \end{eqnarray}

 \noindent
Where $\|\cdot\|=\|\cdot\|_{L^2}$. Let us denote by $\mathfrak{P}$ and $\mathfrak{R}$ the functions defined as
$$
\mathfrak{P}^2:=\|U\|_{\mathcal{H}}\|F\|_{\mathcal{H}}+\frac{1}{|\lambda|}\|F\|_{\mathcal{H}}^2,\qquad \mathfrak{R}^2:=\|U\|_{\mathcal{H}}\|F\|_{\mathcal{H}}+\|F\|_{\mathcal{H}}^2.
$$
We recall also that, in the following three Lemmas, $u$ and $v$ are solutions of  system \eqref{Eqrtt1}.
\begin{lemma} \label{pro12} 
           Over the viscoelastic component $ (\ell_0, \ell_1) $, { the solution of \eqref{reso}  verifies }
            \begin{equation}
            \label{eqlemma2}
                 \int_{\ell_0}^{\ell_1} \left(\alpha |u_{xx}|^2 +  |v|^2 \right)dx \leq \dfrac{C}{|\lambda| } \mathfrak{P},
            \end{equation}
            for $\lambda$ large.
            \end{lemma}

            \begin{proof}
          Using \eqref{Eqrtt1} and \eqref{uno} and recalling the definition of $\mathfrak{R}$ we get
              \begin{equation}\label{duno}
\int_{\ell_0}^{\ell_1}\alpha_0|u_{xx}(x)|^{2} dx+\gamma_2|u(\xi_2)|^2 +\gamma_3|u_{x}(\xi_2)|^2 \leq  \frac{C}{|\lambda|^2} \mathfrak{R}.
    \end{equation}
    From \eqref{uno} and \eqref{v1xx}, and recalling the definition of $\mathfrak{P}$ we get
    \begin{eqnarray}\label{v2xx}
       \|v_2\|_{H^{2}}= \|v_{xx}-v_{1,xx}\| \leq C\, \mathfrak{P}.
    \end{eqnarray}
    Using $\eqref{eqv1}$ and  \eqref{v2xx} 
    we arrive at
    \begin{eqnarray*}
         \|v_2\|_{H^{-2}} & \leq & \dfrac{C}{|\lambda|}\mathfrak{P}.
    \end{eqnarray*}
        On the other hand, by interpolation inequality and Young inequality, we have
        \begin{eqnarray} \label{v2} 
            \|v_2\|^2
            \leq 
            \|v_2\|_{H^2} \|v_2\|_{H^{-2}} 
            \leq  \dfrac{C}{|\lambda| } \mathfrak{P}.
        \end{eqnarray}
        Therefore, 
        \begin{eqnarray*} \label{eqv}
            \int^{\ell_1}_{\ell_0}\left( |v|^2 + \alpha |u_{xx}|^2\right) dx \leq   \dfrac{C}{|\lambda| } \mathfrak{P}\quad \Rightarrow          \int^{\ell_1}_{\ell_0} \left(|v|^2 + \alpha |u_{xx}|^2 \right) dx \leq \dfrac{C}{|\lambda|^{2}}  \|F\|^2_{\mathcal{H}} +  \epsilon \|U\|^2_{\mathcal{H}}.
        \end{eqnarray*}
\end{proof}

\begin{lemma}\label{viscco}
Let us denote by $b$ either $\ell_0$ or $\ell_1$, then the solution of  system \eqref{reso} verifies the following inequalities over  the visco elastic interval $(\ell_0,\ell_1)$. 
\begin{equation}\label{viscbb}
\left|\frac{\alpha u_{xxx}(b)+\alpha_0v_{xxx}(b)}{i\lambda}\right|\leq \frac{c}{|\lambda|^{5/8}}\mathfrak{P},\qquad 
\left|\frac{\alpha u_{xx}(b)+\alpha_0v_{xx}(b)}{i\lambda}\right|\leq \frac{c}{|\lambda|^{7/8}}\mathfrak{P} .
\end{equation}
Moreover, we have 
\begin{eqnarray}\label{Lorder}
    |v(\ell_0^+)|\leq \frac{C}{|\lambda|^{\frac{3}{8}}}\mathfrak{P},\qquad     |v_x(\ell_0^+)|\leq \frac{C}{|\lambda|^{\frac{1}{8}}}\mathfrak{P}.
 \end{eqnarray}
\end{lemma}
\begin{proof}Denoting by $Y$ the function
$$
Y=\frac{1}{i\lambda}\left(\alpha  u_{xx}+\alpha_0 v_{2,xx}\right):= \frac{1}{i\lambda}w_{xx}
$$
where $w=\alpha  u+\alpha_0 v_{2}$, note that 
$$
\|Y\|_{L^2}\leq \frac{c}{|\lambda|}\left(\|u_{xx}\|_{L^2}+\|v_{2,xx}\|_{L^2}\right)\leq \frac{c}{|\lambda|}\mathfrak{P}.
$$
 Since
\begin{eqnarray}\label{Yxx}
Y_{xx}= - v_2\in H^2,\quad\Rightarrow\quad  Y_{xxxx}= - v_{2,xx}\in L^2.
\end{eqnarray}
It follows from Lemma \ref{desGN} that
\begin{eqnarray*}
|Y(b)|
&\leq& C\|Y\|+ C\|Y\|^{3/4}\|Y_{xx}\|^{1/4}\\
&\leq& C\|Y\|+C\|Y\|^{3/4}\|v_2\|^{1/4}\\
&\leq& \frac{c}{|\lambda|}\mathfrak{P}+C(\frac{c}{|\lambda|}\mathfrak{P})^{3/4}(\frac{c}{|\lambda|^{1/2}}\mathfrak{P})^{1/4}\\
&\leq& \frac{c}{|\lambda|^{7/8}}\mathfrak{P}.
\end{eqnarray*}
 Using similar approach 
we have 
\begin{eqnarray}
|Y_x(b)| 
&\leq&  C\|Y\|+C\|Y\|^{1/4}\|Y_{xx}\|^{3/4}\nonumber\\
&\leq& \frac{c}{|\lambda|}\mathfrak{P}+C(\frac{c}{|\lambda|}\mathfrak{P})^{1/4}(\frac{c}{|\lambda|^{1/2}}\mathfrak{P})^{3/4}\nonumber\\
&\leq& \frac{c}{|\lambda|^{5/8}}\mathfrak{P}.\label{tresbb}
\end{eqnarray}
Finally, since 
\begin{eqnarray}\label{tresbb2}
\frac{\alpha u_{xxx}(b)+\alpha_0v_{xxx}(b)}{i\lambda}=Y_x(b)+\frac{\alpha_0v_{1,xxx}(b)}{i\lambda},
\end{eqnarray}
using Lemma \ref{desGN} and inequalities \eqref{v1}, we arrive to 
\begin{eqnarray*}
|v_{1,xxx}(b)|&\leq& c\|v_{1,xx}\|+c\|v_{1,xx}\|^{1/4}\|v_{1,xxxx}\|^{3/4}\\
&\leq& \frac{c}{|\lambda|^{1/2}}\|F\|+\frac{c}{|\lambda|^{1/8}}\|F\|.
\end{eqnarray*}
Hence for $\lambda$ large we get 
\begin{eqnarray*}
\left|\frac{\alpha_0v_{1,xxx}(b)}{i\lambda}\right|&\leq& \frac{c}{|\lambda|}\|v_{1,xx}\|+ \frac{c}{|\lambda|}\|v_{1,xx}\|^{1/4}\|v_{1,xxxx}\|^{3/4}\\
&\leq& \frac{c}{|\lambda|^{9/8}}\|F\|.
\end{eqnarray*}
Inserting the above inequality and inequality \eqref{tresbb} into \eqref{tresbb2} we arrive to inequality \eqref{viscbb}. Using the same above reasoning we get the first inequality in \eqref{viscbb}. 
Finally, using Lemma \ref{desGN} 
\begin{eqnarray*}
    |v_x(\ell_0^+)|\leq  \|v_x\|_{L^2}^{\frac{1}{2}}\|v_{xx}\|_{L^2}^{\frac{1}{2}}\leq  \|v\|_{L^2}^{\frac{1}{4}}\|v_{xx}\|_{L^2}^{\frac{3}{4}}\leq \frac{C}{|\lambda|^{\frac{1}{8}}}\mathfrak{P}.
 \end{eqnarray*}
Similarly we get
    \small \begin{eqnarray*}
        |v(\ell_0^{+})|
        & \leq & C \|v\| + C \|v\|^{\frac{3}{4}} \|v_{xx}\|^{\frac{1}{4}} \leq\dfrac{C}{|\lambda|^{\frac{3}{8}}}\mathfrak{P},
    \end{eqnarray*} \normalsize
from where our conclusion follows.
\end{proof}

 \subsection{Estimates over the elastic component, interval $(0,  \ell_0)$}
 
 The next Lemma is important to apply the observability inequality to get the estimate of the elastic component. Here we follow the same ideas of \cite {LiuLiu}.

\begin{lemma}\label{Funda}
Over the interval $(0,\xi_1)$ the solution of  system \eqref{reso} satisfies 
$$
|u_{xx}(0)|\leq C\frac{|u_{xxx}(0)|}{|\lambda|^{1/2}}+\frac{C}{|\lambda|^{3/4}}\|F\|_{\mathcal{H}}+\frac{Ce^{-|\lambda|^{1/2}\ell_0}}{|\lambda|^{1/2}}\|U\|_{\mathcal{H}}^{1/2}\|F\|_{\mathcal{H}}^{1/2}.
$$
Similarly, over the interval $(\ell_1,\ell)$, we have
$$
|u_{xx}(\ell)|\leq C\frac{|u_{xxx}(\ell)|}{|\lambda|^{1/2}}+\frac{C}{|\lambda|^{3/4}}\|F\|_{\mathcal{H}}+\frac{Ce^{-|\lambda|^{1/2}\ell_0}}{|\lambda|^{1/2}}\|U\|_{\mathcal{H}}^{1/2}\|F\|_{\mathcal{H}}^{1/2}.
$$

\end{lemma}
\begin{proof}
  The model over $(0,\xi_1)$ is written as 
		\begin{align}
		\label{Eqelas1}	i\lambda u-v&=f_{1}\in H^2(0,\xi_1),\\
		\label{Eqelas2}	i\lambda   v +\alpha u_{xxxx}&=  f_{2}\in L^{2}(0,\xi_1).
		\end{align}
Substitution of $v$ given in equation (\ref{Eqelas1}) into (\ref{Eqelas2}) yields
$$
u_{xxxx}-\sigma^4u=G, \quad\mbox{with}\quad G=\frac{1}{\alpha}\left(i\lambda f_1+  f_2\right), \quad \sigma =\left(\frac{ \lambda^2}{\alpha}\right)^{1/4}.
$$
Denoting by $D=\dfrac{d}{dx}$,  the above system can be written as 
\begin{equation*}\label{phidef}
(D^4-\sigma^4)u=G
,\quad \Rightarrow\quad
(D^2+\sigma^2)\underbrace{(D+\sigma)(D-\sigma)u}_{:=\varphi}=G.
\end{equation*}
So we have that 
$
(D^2+\sigma^2)\varphi=G.
$
Solution to this second order linear nonhomogeneous equation is
\begin{eqnarray*}
\varphi(x)&=&C_1e^{i\sigma x}+C_2e^{-i\sigma x}+\int_0^xe^{-i\sigma s}G(s)ds\; \frac{e^{i\sigma x}}{2i\sigma}-\int_0^xe^{i\sigma s}G(s)ds\;\frac{e^{-i\sigma x}}{2i\sigma}\nonumber\\
&:=& J_1+J_2+J_3+J_4\label{1234}
\end{eqnarray*}
where 
$$
C_1=\frac{1}{2i\sigma}\left[i\sigma\varphi(0^{+})+\varphi_x(0^{+})
\right],\quad C_2=\frac{1}{2i\sigma}\left[i\sigma\varphi(0^{+})-\varphi_x(0^{+})
\right].
$$
Since $\varphi=u_{xx}-\sigma^2u$ and  $u(0)=u_x(0)=0$ we get 
$\varphi(0)=u_{xx}(0)$, $\varphi_x(0)=u_{xxx}(0)$. Hence
\begin{equation}\label{ccc}
C_1=\frac{1}{2i\sigma}\left[i\sigma u_{xx}(0)+u_{xxx}(0)
\right],\quad C_2=\frac{1}{2i\sigma}\left[i\sigma  u_{xx}(0)-u_{xxx}(0)
\right].
\end{equation}
Denoting by $z=(D+\sigma)u$ we have that 
$$
z_x-\sigma z=\varphi,\quad \Rightarrow\quad 
z(x)=e^{\sigma (x-\xi_1)}z(\xi_1^{-})+\int_{\xi_1}^xe^{\sigma (x-s)}\varphi(s)\;ds, 
$$
where $z(\xi_1^-)=u_x(\xi_1^-)-\sigma u(\xi_1^-)$. Since $z(0)=0$, then 
\begin{equation}\label{chave}
0=e^{-\sigma \xi_1}z(\xi_1^{-})-\int_0^{\xi_1}e^{-\sigma s}\left[J_1(s)+J_2(s)+J_3(s)+J_4(s)\right]\;ds.
\end{equation}
Note that in point $x=\xi_1$ we have that   $$|u_x({\xi_1})|\leq c\|u\|^{1/4}\|u_{xx}\|^{3/4}\leq \frac{c}{|\lambda|^{1/4}}\left(\|U\|+\|U\|_{\mathcal{H}}^{1/2}\|F\|_{\mathcal{H}}^{1/2}\right),$$
and
	\begin{align*}\label{viscbb2}
|z(\xi_1)|=|u_x(\xi_1^-)-\sigma u(\xi_1^-)|=|u_x(\xi_1^+)-\sigma u(\xi_1^+)|
&\leq \frac{C}{|\lambda|^{7/8}}\|U\|_{\mathcal{H}}^{1/2}\|F\|_{\mathcal{H}}^{1/2}+\frac{C}{|\lambda|^{1/2}}\|F\|_{\mathcal{H}}.
\end{align*}
Thus, 
$$
|e^{-\sigma \xi_1}z(\xi_1^{-})|\leq  \frac{Ce^{-\sigma \xi_1}}{|\lambda|^{1/2}}\left(\|U\|^{1/2}\|F\|^{1/2}+\|F\|\right).
$$
Note that 
\begin{eqnarray*}
\int_0^{\xi_1}e^{-\sigma s}J_1(s)\;ds
=C_1\int_0^{\xi_1}e^{\sigma (-1+i)s}\;ds
=\frac{C_1}{\sigma (-1+i)}\left(e^{-\xi_1\sigma+i\xi_1\sigma}-1\right).
\end{eqnarray*}
Similarly, we have 
\begin{eqnarray*}
\int_0^{\xi_1}e^{-\sigma s}J_2(s)\;ds
&=&-\frac{C_2}{\sigma (1+i)}\left(e^{-\xi_1\sigma-i\xi_1\sigma}-1\right).
\end{eqnarray*}
Therefore, 
\begin{eqnarray*}
\int_0^{\xi_1}e^{-\sigma s}(J_1(s)+J_2(s))\;ds
&=&K_2-\frac{C_1}{\sigma (-1+i)}+\frac{C_2}{\sigma (1+i)},\\
\end{eqnarray*}
where 
$$
K_2=\frac{C_1e^{-\xi_1\sigma+i\xi_1\sigma}}{\sigma (-1+i)}-\frac{C_2e^{-\xi_1\sigma-i\xi_1\sigma}}{\sigma (1+i)},\quad\Rightarrow\quad |K_2|\leq \frac{ce^{-\xi_1\sigma}}{\sigma}\left(\sigma|u_{xx}(0^+)|+|u_{xxx}(0^+)|\right).
$$
The last inequality is due to  (\ref{ccc}). Moreover,
\begin{eqnarray}
\int_0^{\xi_1}e^{-\sigma s}(J_1+J_2)\;ds&=&K_2-\frac{i\sigma u_{xx}(0)+u_{xxx}(0)}{2i\sigma^2 (-1+i)}+\frac{i\sigma u_{xx}(0)-u_{xxx}(0)
}{2i\sigma^2 (1+i)}\nonumber\\
&=&K_2+\frac{u_{xx}(0^{+})}{2i\sigma}-\frac{u_{xxx}(0^{+})
}{2\sigma^2}.\label{pripart}
\end{eqnarray}
On the other hand, 
\begin{eqnarray}
\int_0^{\xi_1}e^{-\sigma s}J_3(s)\;ds&=&\frac{1}{2i\sigma}\int_0^{\xi_1}e^{-\sigma (1-i)s}\int_0^se^{-i\sigma \tau}G(\tau)d\tau\; \;ds\nonumber\\
&=&\frac{1}{-2i\sigma^2 (1-i)}\int_0^{\xi_1}\int_{0}^se^{-i\sigma \tau}G(\tau)d\tau\; \;d\left(e^{-\sigma (1-i)s}\right)\nonumber\\
&=&\frac{-e^{-\xi_1\sigma+i\xi_1\sigma}}{2i\sigma^2 (1-i)}\int_0^{\xi_1}e^{-i\sigma \tau}G(\tau)d\tau +\frac{1}{2i\sigma^2 (1-i)}\int_0^{\xi_1}e^{-\sigma s}G(s)\;ds.
\label{j3}
\end{eqnarray}
Since$f_1(0)=f_1'(0)=0$, performing integration by parts leads to
\begin{eqnarray*}
& &\int_0^{\xi_1}e^{-\sigma s}G(s)\;ds=\frac{1}{\alpha}\int_0^{\xi_1}e^{-\sigma s}\left[i\lambda f_1+f_2\right]\;ds\\
&=&\underbrace{-\frac{i\lambda }{\alpha\sigma^2}e^{-\sigma \xi_1}\left[\sigma f_1(\xi_1)+f_{1,s}(\xi_1)\right]}_{:=X_1}+\underbrace{\frac{i\lambda }{\alpha\sigma^2}\int_0^{\xi_1}e^{-\sigma s}f_{1,ss}(s)\;ds+\frac{1}{\alpha}\int_0^{\xi_1}e^{-\sigma s}f_2(s)\;ds}_{:=X_2}\\
\end{eqnarray*}
Note that 
$$
\left|X_1\right|\leq c\sigma\|F\|e^{-\sigma \xi_1},\quad \left|X_2\right|\leq \frac{C}{\sqrt{\sigma}}[1-e^{-\xi_1\sigma}]^{1/2}\left(\|f_1\|_{H^2}+\|f_2\|\right).
$$
Substitution into (\ref{j3}) we get 
\begin{eqnarray}
\left|\int_0^{\xi_1}e^{-\sigma s}J_3\;ds\right|&\leq &\frac{C}{\sigma^{5/2}}\left(\|f_1\|_{H^2}+\|f_2\|\right),
\label{j4}
\end{eqnarray}
for $\lambda$ large. Similarly, we have 
\begin{eqnarray}
\left|\int_0^{\xi_1}e^{-\sigma s}J_4\;ds\right|&\leq &\frac{C}{\sigma^{5/2}}\left(\|f_1\|_{H^2}+\|f_2\|\right).
\label{j5}
\end{eqnarray}
Substitution of inequalities (\ref{pripart}), (\ref{j4}) and (\ref{j5})  into (\ref{chave}) yields 
$$
|u_{xx}(0)|\leq C\frac{|u_{xxx}(0)|}{\sigma}+\frac{C}{\sigma^{3/2}}\left(\|f_1\|_{H^2}+\|f_2\|\right) ,
$$
because 
$
|\sigma K_2|\leq \epsilon |u_{xx}(0)|+\epsilon\frac{|u_{xxx}(0)|}{\sigma},
$
for $\lambda$ large ($\sigma=\sqrt[4]{1 /\alpha}\sqrt{|\lambda|}$). The proof is now complete. 
\end{proof}


\subsection{Estimates over the elastic components}
The following estimates will be very useful in what follows.
\begin{lemma} \label{util}
Let us denote by $(u,v)$  the solution of \eqref{Eqrtt1}--\eqref{Eqrtt2} then over the elastic components $(0,\ell_0)\cup (\ell_1,\ell)$ , the following estimates hold, 
\begin{eqnarray*}
        \|v_x\|  & \leq &  C |\lambda|^{\frac{1}{2}}\|U\|_{\mathcal{H}}+C\|U\|_{\mathcal{H}}^{\frac{1}{2}} \|F\|_{\mathcal{H}}^{\frac{1}{2}},\\
                \|u_{xxx}\|  & \leq &  C |\lambda|^{\frac{1}{2}}\|U\|_{\mathcal{H}}+C\|U\|_{\mathcal{H}}^{\frac{1}{2}} \|F\|_{\mathcal{H}}^{\frac{1}{2}}.  
    \end{eqnarray*}  
    \end{lemma} 
  \begin{proof}
    Using  $i\lambda u - v = f_1 $ over $(0,\xi_1)\cup(\xi_1,\ell_0)$ or $(\ell_1,\ell)$ we arrive to 
     \begin{eqnarray*}
        \|v_x\| \leq C \|v\|^{\frac{1}{2}} \|v_{xx}\|^{\frac{1}{2}} 
        \leq  C \|v\|^{\frac{1}{2}} \|i \lambda u_{xx} + f_{1,xx}\|^{\frac{1}{2}} 
        & \leq &  C |\lambda|^{\frac{1}{2}}\|U\|_{\mathcal{H}}+\|v\|^{\frac{1}{2}} \|F\|_{\mathcal{H}}^{\frac{1}{2}},  
    \end{eqnarray*}
    from where the first inequality holds. Using interpolation we get 
     \begin{eqnarray*}
        \|u_{xxx}\| \leq C \|u_{xx}\|^{\frac{1}{2}} \|u_{xxxx}\|^{\frac{1}{2}} 
        \leq  C \|u_{xx}\|^{\frac{1}{2}} \|i \lambda v + f_{2}\|^{\frac{1}{2}} 
        & \leq &  C |\lambda|^{\frac{1}{2}}\|U\|_{\mathcal{H}}+\|v\|^{\frac{1}{2}} \|F\|_{\mathcal{H}}^{\frac{1}{2}}.  
    \end{eqnarray*}
    Then, the second inequality holds.
  \end{proof}
\begin{lemma} \label{nuf}
Let us denote by $[a,b]$ any of the elastic  intervals $[0,\ell_0]$, $[\ell_1,\ell]$ and let $q$ bounded $C^1$-function, then  the solution of \eqref{Eqrtt1}--\eqref{Eqrtt2} satisfies 
   $$ \left|\int_a^b qf_2 \overline{u_x} dx \right|+\left|\int_a^b qf_{1,x}  \overline{v} dx \right|\leq \dfrac{C}{|\lambda|}\|F\|^2_{\mathcal{H}}+\epsilon \|U\|^2_{\mathcal{H}}.    $$
  \end{lemma} 
  \begin{proof}
     We begin by considering the interval \((0, \ell_0)\). Utilizing Lemma \ref{util} and equation \eqref{Eqrtt1}, we can establish the following inequality 

    \begin{eqnarray*}
        \left| \int^{\ell_0}_0 f_2 q \overline{u_x} dx \right| & = & \left| \dfrac{1}{i \lambda} \int^{\ell_0}_0 f_2q \overline{[v_x + f_{1,x}]} dx\right| \leq  \dfrac{C}{|\lambda|}\|F\|_{\mathcal{H}} \|v_x\| + \dfrac{C}{|\lambda|}\|F\|^2_{\mathcal{H}} \\
        & \leq & \dfrac{C}{|\lambda|^{\frac{1}{2}}}\|F\|_{\mathcal{H}}\|U\|_{\mathcal{H}}+  \dfrac{C}{|\lambda|}\|F\|^2_{\mathcal{H}}.
    \end{eqnarray*}
From where the first inequality follows. Let us consider the second inequality, integration over $(0,\xi_1)\cup(\xi_1,\ell_0)$ yields 
 \begin{eqnarray}
    \int^{\ell_0}_0 qf_{1,x}  \overline{v} dx  
    &=& -\dfrac{1}{i \lambda} \int^{\ell_0}_0 qf_{1,x} \overline{i \lambda v} dx\nonumber \\
    &=&\dfrac{1}{i\lambda} \int^{\ell_0}_0 qf_{1,x} (\overline{i \alpha u_{xxxx}-f_2}) dx \nonumber \\
    &=& -\dfrac{1}{i \lambda} \int^{\ell_0}_0 q f_{1,x} \overline{f_2} dx + \dfrac{\alpha}{i \lambda} \int^{\ell_0}_0 f_{1,x} q \overline{u_{xxxx}} dx\nonumber \\
    &=& -\dfrac{1}{i \lambda} \int^{\ell_0}_0 q f_{1,x} \overline{f_2} dx 
    + { \dfrac{\alpha}{i \lambda} f_{1,x}(\ell_0^{-}) \overline{u_{xxx}}(\ell_0^{-}) } + \dfrac{\alpha \gamma_0}{i \lambda } f_{1,x}(\xi_1) v(\xi_1)\nonumber \\
    && - \dfrac{\alpha}{i \lambda} \int^{\ell_0}_0 (f_{1,x} q )_x\overline{u_{xxx}} dx ,\label{xx11}
\end{eqnarray} 
where we used that $f_{1,x}(0)=0$. 
    By  Lemma \ref{viscco} and Lemma \ref{util} we have 
    \begin{eqnarray}\label{M7}
        \bigg| \dfrac{\alpha}{i \lambda} f_{1,x}(\ell_0^{-}) \overline{u_{xxx}}(\ell_0^{-}) \bigg| &\leq& \dfrac{C}{|\lambda|^{\frac{5}{8}}}\mathfrak{P}\|F\|_{\mathcal{H}}\leq \epsilon \|U\|_{\mathcal{H}}^2+\dfrac{C}{|\lambda|^{\frac{5}{6}}}\|F\|_{\mathcal{H}}^2
    \end{eqnarray}
    \begin{eqnarray}\label{M8}
        \bigg|  \dfrac{\alpha}{i \lambda} \int^{\ell_0}_0 (f_{1,x} q )_x\overline{u_{xxx}} dx \bigg| &\leq& \dfrac{C}{|\lambda|}\|u_{xxx}\|\|F\|_{\mathcal{H}}\leq \epsilon \|U\|_{\mathcal{H}}^2+\dfrac{C}{|\lambda|}\|F\|_{\mathcal{H}}^2
    \end{eqnarray}
    Using \eqref{uno} we get
    \begin{eqnarray}\label{M9}
        \bigg|   \dfrac{\alpha \gamma_0}{i \lambda } f_{1,x}(\xi_1) v(\xi_1) \bigg| &\leq& \dfrac{C}{|\lambda|}\|F\|_{\mathcal{H}}(\|U\|_{\mathcal{H}}^{1/2}\|F\|_{\mathcal{H}}^{1/2})\leq \epsilon \|U\|_{\mathcal{H}}^2+\dfrac{C}{|\lambda|^{4/3}}\|F\|_{\mathcal{H}}^2,
    \end{eqnarray}
 for $\lambda$ large. Substitution of  \eqref{M7}, \eqref{M8}, \eqref{M9}  into \eqref{xx11} yields the requerid inequality. 
 Analogously for $(\ell_1,\ell)$.
\end{proof}  
    Let us introduce 
    \begin{eqnarray*}
        I(x) =  \dfrac{1}{2} \alpha |u_{xx}(x)|^2 + \dfrac{1}{2}|v(x)|^2 .
    \end{eqnarray*}
\begin{lemma} \label{lemma:obs}
   The solution of \eqref{Eqrtt1}--\eqref{Eqrtt2} satisfies 
    \begin{eqnarray*}
\bigg|  \ell_0 I(\ell_0^{-}) - \int_0^{\ell_0} I(x) dx-  \int_{0}^{\ell_0} \alpha |u_{xx}|^2 dx \bigg| &\leq & \dfrac{C}{|\lambda|}\|F\|^2_{\mathcal{H}}+\epsilon \|U\|^2_{\mathcal{H}}. \\
       |\lambda|^{\frac{1}{4}} \bigg|  (\ell-\ell_1)I(\ell_1^{+}) - \int_{\ell_0}^\ell I(x) dx-  \int_{\ell_0}^\ell \alpha |u_{xx}|^2 dx \bigg| &\leq & C \,\mathfrak{R}^2.
    \end{eqnarray*}
\end{lemma}
    \begin{proof}
Multiplying equation \eqref{Eqrtt2} by  $ q(x) \overline{u_x}$, with $q=x$ and integrating over $]0,\xi_1[\cup]\xi_1,\ell_0[$ we get  
   \begin{eqnarray}
      \frac{1}{2}\int_0^{\ell_0}I(x)dx +\int_0^{\ell_0}\alpha |u_{xx}|^2 dx&=&   \frac{\ell_0}{2}I(\ell_0)
      - \underbrace{\alpha q(\xi_1) \salto{ u_{xxx}}{1} \overline{u_x}(\xi_1)}_{X_1}
    \nonumber \\
     && +\underbrace{\alpha (\ell_0  u_{xxx}(\ell_0^-)-u_{xx}(\ell_0^-)) \overline{u_x}(\ell_0)}_{X_2}    \nonumber\\
     && + \underbrace{\int_{0}^{\ell_0}  f_2 q \overline{u_x} dx+  \int_{0}^{\ell_0}  vq \overline{f_{1,x}}dx}_{X_3}.\label{III}
   \end{eqnarray} 
Combining the transmission conditions at $\ell_0$ with equation~\eqref{Eqrtt2}, we obtain
   \begin{eqnarray*}
X_2&=&\alpha (\ell_0  w_{xxx}(\ell_0)-w_{xx}(\ell_0)) \overline{u_x}(\ell_0)\\
 &=&\frac{\alpha }{i\lambda}(\ell_0  w_{xxx}(\ell_0)-w_{xx}(\ell_0)) \overline{v_x}(\ell_0)+\frac{\alpha }{i\lambda}(\ell_0  w_{xxx}(\ell_0)-w_{xx}(\ell_0)) \overline{f_{1,x}}(\ell_0)
   \end{eqnarray*} 
   Application of Lemma \ref{viscco}  yields the following inequality 
   \begin{eqnarray*}
\left| X_2\right|&\leq &\frac{c}{|\lambda|^{5/8}}\mathfrak{P}^2\leq  \dfrac{C}{|\lambda|^{5/4}}\|F\|^2_{\mathcal{H}}+\epsilon \|U\|^2_{\mathcal{H}}.
   \end{eqnarray*} 
To estimate $X_1$ first note that 
$$
|u_x(\xi_1)|\leq c\|u_x\|^{1/2}\|u_{xx}\|^{1/2}\leq c\|u\|^{1/4}\|u_{xx}\|^{3/4}\leq \frac{c}{|\lambda|^{1/4}}\|v\|^{1/4}\|u_{xx}\|^{3/4}+ \frac{c}{|\lambda|^{1/4}}\|f\|^{1/4}\|u_{xx}\|^{3/4}
$$
Using that $\salto{ u_{xxx}}{1} =\gamma_1v(\xi_1)$ and relation   \eqref{uno} we get
   \begin{eqnarray*}
  \left| X_1\right|= |\alpha q(\xi_1) \salto{ u_{xxx}}{1} \overline{u_x}(\xi_1)|&\leq &  c\left(\|U\|_{\mathcal{H}}\|F\|_{\mathcal{H}}\right)^{1/2}|u_x(\xi_1)|\\
   &\leq &   c\left(\|U\|_{\mathcal{H}}\|F\|_{\mathcal{H}}\right)^{1/2}|\frac{c}{|\lambda|^{1/4}}\left(\|U\|_{\mathcal{H}}
   +\|U\|_{\mathcal{H}}^{3/4}\|F\|_{\mathcal{H}}^{1/4}\right)\\
   &\leq &    \dfrac{C}{|\lambda|^{2/3}}\|F\|^2_{\mathcal{H}}+\epsilon \|U\|^2_{\mathcal{H}},
   \end{eqnarray*} 
for $\lambda$ large. So, using the above inequalities into \eqref{III} we have that 
   \begin{eqnarray}\label{ffqq}
\left|-  \frac{\ell_0}{2}I(\ell_0)   +   \frac{1}{2}\int_0^{\ell_0}I(x)dx +\int_0^{\ell_0}\alpha |u_{xx}|^2 dx\right|&\leq &  
   \dfrac{C}{|\lambda|^{2/3}}\|F\|^2_{\mathcal{H}}+\epsilon \|U\|^2_{\mathcal{H}}.
          \end{eqnarray} 
Applying the same above procedure over the interval $(\ell_1,\ell)$ our conclusion follows. 
    \end{proof}
\begin{lemma} \label{lemma:obs2}
   The solution of \eqref{Eqrtt1}--\eqref{Eqrtt2} satisfies 
    \begin{eqnarray*}
 \bigg| I(\ell_0)- I(0) \bigg| \leq  \dfrac{C}{|\lambda|^{2/3}}\|F\|^2_{\mathcal{H}}+\epsilon \|U\|^2_{\mathcal{H}},\quad 
  \bigg| I(\ell_1)- I(\ell) \bigg| \leq  \dfrac{C}{|\lambda|^{2/3}}\|F\|^2_{\mathcal{H}}+\epsilon \|U\|^2_{\mathcal{H}}.
    \end{eqnarray*}
\end{lemma}
    \begin{proof}
As in Lemma \ref{lemma:obs} let us multiply equation \eqref{Eqrtt2} by  $\overline{u_x}$ and integrating over $]0,\ell_0[$ we get 
   \begin{eqnarray*}
 \frac{1}{2}I(0)-\frac{1}{2}I(\ell_0)&=&  
      - \alpha q(\xi_1) \salto{ u_{xxx}}{1} \overline{u_x}(\xi_1)
      +\alpha (\ell_0  u_{xxx}(\ell_0^-)-u_{xx}(\ell_0^-)) \overline{u_x}(\ell_0)\\
      && + \int_a^{b}  f_2 q \overline{u_x} dx+  \int_{0}^{\ell_0} vq \overline{f_{1,x}}dx.
   \end{eqnarray*} 
Using the same procedure as in Lemma \ref{lemma:obs}  the first inequality follows. Repeating the same procedure over the interval $(\ell_1,\ell)$, the second inequality follows. 
    \end{proof}
\begin{lemma} \label{lema_aju2}
       The solution of \eqref{Eqrtt1}--\eqref{Eqrtt2} satisfies 
    \begin{eqnarray*}
        \dfrac{1}{|\lambda|}|u_{xxx}(0)|^2 \leq \dfrac{C}{|\lambda|^{1/2}}\|F\|^2_{\mathcal{H}}+\epsilon \|U\|^2_{\mathcal{H}},\quad   \dfrac{1}{|\lambda|}|u_{xxx}(\ell)|^2 \leq \dfrac{C}{|\lambda|^{1/2}}\|F\|^2_{\mathcal{H}}+\epsilon \|U\|^2_{\mathcal{H}}.
    \end{eqnarray*}
    \end{lemma}
    \begin{proof}
    Multiplying \eqref{Eqrtt2} by $ \overline {u_ {xxx}} $ and integrating over $]0,\xi_1[$ and $]\xi_1,\ell_0[$ and summing up the product result and recalling that $\salto{u_{xx}}{1}=0$ and $v_x(0)=0$  we get
        \begin{eqnarray*}
          &&  [ i \lambda v \overline{u_{xx}}]_{0}^{\ell_0} +\dfrac{\alpha }{2}|u_{xxx}(\xi_1^-)|^2-\dfrac{\alpha }{2}|u_{xxx}(0)|^2+\dfrac{\alpha }{2}|u_{xxx}(\ell_0^-)|^2-\dfrac{\alpha }{2}|u_{xxx}(\xi_1^+)|^2+\dfrac{\alpha }{2}|v_{x}(\ell_0^-)|^2\\
            &&
           = \int_0^{\ell_0} f_2 \overline{u_{xxx}} dx + \int_0^{\ell_0} v_x \overline{f_{1,xx}} dx.
        \end{eqnarray*}
        Thus,
         \begin{eqnarray*}
          \dfrac{\alpha }{2}|u_{xxx}(0)|^2&=& \underbrace{\dfrac{\alpha }{2}|u_{xxx}(\ell_0^-)|^2+\dfrac{\alpha }{2}|v_{x}(\ell_0^-)|^2+ [ i \lambda v \overline{u_{xx}}]_{0}^{\ell_0}}_{=X_3} +\underbrace{\dfrac{\alpha }{2}|u_{xxx}(\xi_1^-)|^2-\dfrac{\alpha }{2}|u_{xxx}(\xi_1^+)|^2}_{=X_4}\\
            &&
           \underbrace{- \int_0^{\ell_0} f_2 \overline{u_{xxx}} dx -\int_0^{\ell_0}  v_x \overline{f_{1,xx}} dx}_{=X_5}.
        \end{eqnarray*}
        Therefore we have that 
         \begin{eqnarray}\label{u3(0)}
          \dfrac{\alpha }{2|\lambda|}|u_{xxx}(0)|^2&=& \frac{1}{|\lambda|}X_3+ \frac{1}{|\lambda|}X_4+ \frac{1}{|\lambda|}X_5.
        \end{eqnarray}
Using the transmission conditions and Lemma \ref{viscco}
         \begin{eqnarray*}
|X_3|&=& |\underbrace{\dfrac{\alpha }{2}|w_{xxx}(\ell_0^+)|^2}_{\leq {C}{|\lambda|^{\frac{6}{8}}}\mathfrak{P}^2}+\underbrace{\dfrac{\alpha }{2}|v_{x}(\ell_0^-)|^2}_{\leq {C}{|\lambda|^{-\frac{1}{4}}}\mathfrak{P}^2}+ \underbrace{[ i \lambda v \overline{u_{xx}}]_{0}^{\ell_0}}_{\leq {C}{|\lambda|^{\frac{6}{8}}}\mathfrak{P}^2}|
\leq  C|\lambda|^{\frac{6}{8}}\mathfrak{P}^2.
        \end{eqnarray*}
        So, we have that 
        $$
         \frac{1}{|\lambda|}X_3\leq \dfrac{C}{|\lambda|^{1/2}}\|F\|^2_{\mathcal{H}}+\epsilon \|U\|^2_{\mathcal{H}}.
         $$
        Using that $|u_{xxx}(\xi_1^+)|\leq c\|u_{xxx}\|^{1/2}\|u_{xxxx}\|^{1/2}\leq c|\lambda|^{3/4}\|u_{xx}\|^{1/4}\|v\|^{3/4}$ we obtain
         \begin{eqnarray*}
|X_4|&=&\frac \alpha 2\left|[u_{xxx}(\xi_1^-)-u_{xxx}(\xi_1^+)][u_{xxx}(\xi_1^-)+u_{xxx}(\xi_1^+]\right|\\
        &=&\frac \alpha 2\left|\salto{u_{xxx}}{1}[\salto{u_{xxx}}{1}-u_{xxx}(\ell_0^-)]\right|\\
        &\leq &\frac \alpha 2 \salto{u_{xxx}}{1}^2 +\left|\salto{u_{xxx}}{1}u_{xxx}(\xi_1^-)\right|\\
        &\leq &c\mathfrak{R}^2 +c|\lambda|^{3/4}\|U\|^{3/2}\|F\|^{1/2}.
        \end{eqnarray*}
               Then we find
        $$
         \frac{1}{|\lambda|}X_4\leq \dfrac{C}{|\lambda|}\|F\|^2_{\mathcal{H}}+\epsilon \|U\|^2_{\mathcal{H}}.
         $$
Using Lemma \ref{util}  we get
       \begin{eqnarray*}
|X_5|
        &\leq &c|\lambda|^{1/2}\mathfrak{R}^2,
        \end{eqnarray*}
                and  
        $$
         \frac{1}{|\lambda|}X_5\leq \dfrac{C}{|\lambda|}\|F\|^2_{\mathcal{H}}+\epsilon \|U\|^2_{\mathcal{H}}.
         $$
Substituting  the corresponding inequalities of  $X_3$, $X_4$ and  $X_5$ into \eqref{u3(0)},
 our conclusion follows. To get the second inequality, we can follow the same procedure over the interval $(\ell_1,\ell)$. 
    \end{proof}

\begin{thm}\label{Teo2} 
      Under the above conditions, we have that the semigroup $(e^{\mathcal{A}t})_{t \ge 0}$ associated to system \eqref{eq21}--\eqref{wxx}  is of Gevrey-4-class. 
    \end{thm}
\begin{proof}
Application of Lemma \ref{Funda} yields the following inequality 
$$
I(0)\leq \frac{c}{|\lambda|}|u_{xxx}(0)|^2+\frac{c}{|\lambda|^{3/2}}\mathfrak{R}^2,\quad 
I(\ell)\leq \frac{c}{|\lambda|}|u_{xxx}(\ell)|^2+\frac{c}{|\lambda|^{3/2}}\mathfrak{R}^2.
$$
Applying Lemma~\ref{lema_aju2}, we obtain
\[
I(0)\leq \frac{C}{|\lambda|^{1/2}}\|F\|_{\mathcal{H}}^2 + \epsilon \|U\|_{\mathcal{H}}^2, 
\quad 
I(\ell)\leq \frac{C}{|\lambda|^{1/2}}\|F\|_{\mathcal{H}}^2 + \epsilon \|U\|_{\mathcal{H}}^2.
\]
Next, by Lemma~\ref{lemma:obs2}, we have
\[
I(\ell_0)\leq \frac{C}{|\lambda|^{1/2}}\|F\|_{\mathcal{H}}^2 + \epsilon \|U\|_{\mathcal{H}}^2, 
\quad 
I(\ell_1)\leq \frac{C}{|\lambda|^{1/2}}\|F\|_{\mathcal{H}}^2 + \epsilon \|U\|_{\mathcal{H}}^2.
\]
Finally, invoking Lemma~\ref{pro12} together with Lemma~\ref{lemma:obs}, we conclude that
\[
\|U\|_{\mathcal{H}}^2 \leq \frac{C}{|\lambda|^{1/2}}\|F\|_{\mathcal{H}}^2 + \epsilon \|U\|_{\mathcal{H}}^2,
\]
from which the desired estimate follows. Indeed, there exists a constant \( C > 0 \) such that
\begin{equation}\label{final}
    |\lambda|^{1/4} \, \|(i\lambda - \mathcal{A})^{-1}\|_{\mathcal{L}(\mathcal{H})} \leq C.
\end{equation}

\end{proof}

\begin{rem}
As a direct consequence we get the semigroup $(e^{\mathcal{A}t})_{t \ge 0}$ is instantaneous differentiable \cite{Crand-Pazy} and immediately norm continuous
(immediately uniformly continuous) \cite{davies}. Moreover since $0\in\varrho(\mathcal{A})$ inequality \eqref{final} implies the exponential stability which together with 
the norm continued property implies the spectrum-determined growth property \cite[p. 299]{engel}, that means that 
\[
 \omega_0 := \inf\{\omega \in \mathbb{R}:\quad \exists M_\omega \ge 1 \text{ such that }\|e^{\mathcal{A}t}\|\le M_\omega\,e^{\omega t}\quad \forall t \ge 0\} =\sup\{\operatorname{Re}\lambda: \quad \lambda \in \sigma(\mathcal{A})\}.
\]
\end{rem}
\centerline{\textsc{Funding}}

{\small J.E. Mu\~{n}oz Rivera would like to thank CNPq project 307947/2022-0 for the financial support. Project Fondecyt 1230914.
M.G. Naso has been partially supported by Gruppo Nazionale per la Fisica Matematica (GNFM) of the Istituto Nazionale di Alta Matematica
(INdAM).}

\medskip

\centerline{\textsc{Disclosure statement}}

{\small This work does not have any conflict of interest.}

%

%
%
\providecommand{\bysame}{\leavevmode\hbox to3em{\hrulefill}\thinspace}
\providecommand{\MR}{\relax\ifhmode\unskip\space\fi MR }
\providecommand{\MRhref}[2]{%
  \href{http://www.ams.org/mathscinet-getitem?mr=#1}{#2}
}
\providecommand{\href}[2]{#2}

\end{document}